\documentclass{article}
\usepackage{graphicx} % Required for inserting images
\usepackage[usenames,dvipsnames]{xcolor}
\usepackage{tkz-berge}
\usetikzlibrary{fit,shapes}
\usetikzlibrary{positioning,shapes,fit,arrows}
\usepackage{amsmath,amssymb}
\usepackage{amsthm}
\usepackage{enumitem}
\usepackage{graphicx}
\usepackage{caption}
\usepackage{epstopdf}
\usepackage{setspace}
\usepackage{blindtext}
\usepackage{multicol}
\usepackage{tikz}
\usepackage[english]{babel}
\usepackage{float}
\usepackage{empheq}
\usepackage{chngcntr}
\usepackage{subfigure}
\usepackage{cite}
\counterwithout{equation}{section}
\usepackage[retainorgcmds]{IEEEtrantools}
\usepackage{mathptmx}
\usepackage[top=1in,bottom=1in,right=1in,left=1in]{geometry}
\usepackage{comment}
\newtheorem{theorem}{Theorem}[section]

\newtheorem{lemma}[theorem]{Lemma}
\newtheorem{observation}{Observation}
\newtheorem{claim}{Claim}
\newtheorem{definition}{Definition}[section]

\newcommand{\p}{\mathcal{P}}
\theoremstyle{plain}
\newtheorem{thm}{Theorem}%[section]-

\newtheorem{conj}[thm]{Conjecture}
\usepackage{hyperref}
\usepackage{relsize} %to enlarge the size of the summation
\hypersetup{colorlinks,citecolor=blue}
\hypersetup{colorlinks=true,linkcolor=red,filecolor=magenta,    urlcolor=blue}

\theoremstyle{definition}

\theoremstyle{remark}

\title{ Gallai's Path Decomposition of Levi Graphs}

\usepackage{authblk}
\title{Gallai's Path Decomposition of Levi Graph}
\author[1]{Akankshya Sahu }
\author[2]{Sajith Padinhatteeri }
\affil[1,2]{Department of Mathematics, Birla Institute of Technology and Science-Pilani,Hyderabad Campus, Hyderabad-500078, India.}

\date{\today}

\begin{document}

\maketitle

\begin{abstract}
    Gallai's path decomposition conjecture states that for a connected graph $G$ on $n$ vertices, there exists a path decomposition of size $\lceil \frac{n}{2} \rceil$. The Levi graph of order one, denoted by $L_{1}(m,k)$, is a bipartite graph with vertex partition $(A,B)$, where $A$ is the collection of all $(k-1)$-element subsets of $[m]$, and $B$ is the collection of all $k$-element subsets of $[m]$. In this graph, a $(k-1)$-element subset is adjacent to a $k$-element subset if and only if it is properly contained within the $k$-element subset. The path number of a graph $G$ is the minimum size of its path decomposition. Gallai's conjecture can be seen as a conjecture on the upper bound of the path number of a connected graph. In this work, we prove the conjecture for $L_{1}(m,k)$ for all $m \ge 2 $ and $2 \le k \le m$. Moreover, we determine the path number of $L_{1}(m,2)$ for all $m$. 
\end{abstract}

\textbf{Keyword:} Path decomposition, Gallai's conjecture, Levi graph, Bipartite graph, Path number.

\section{Introduction} \label{section 1}

Let $G=(V, E)$ be a finite, simple, undirected, connected graph, where $V=V(G)$ is the set of vertices and $E=E(G)$ is the set of edges of the graph $G$. For an integer $k > 0$, we denote the set $\{1,2,3 \dots, k\}$ by $[k]$. A path is an alternating sequence of distinct vertices and edges, written as $v_1 e_{12} v_2 e_{23} v_3 \dots v_t$ of $G$, such that $v_i$ is adjacent to $v_{i+1}$ through the edge $e_{i(i+1)}$ for all $1 \leq i \leq t-1$. For simplicity of notation, we use $v_1, v_2, v_3 \dots v_t$ or $P_i$ instead of the above notation for a path. 

A \textbf{path decomposition} of a graph $G$ is a collection of paths $\mathcal{P} = \{P_1, P_2, \dots, P_K\} $ such that $\cup_{i} E(P_i) = E(G)$ and $E(P_i)\cap E(P_j) = \phi$ for $i \neq j$. That is, for each edge $e \in E(G)$ there exists a unique path in $\mathcal{P}$ containing $e$. The number of elements in $\mathcal{P}$ of a graph $G$ is called the \textbf{size of the path decomposition} and the minimum size among all possible path decompositions of $G$ is called the \textbf{path number}. It is denoted by $|\p|$ and $pn(G)$, respectively. In general, a decomposition $\mathcal{D}$ of a graph $G$ is a collection of subgraphs of $G$ such that each edge of $G$ belongs to exactly one subgraph in $\mathcal{D}$. If such a decomposition exists, we say that $G$ is covered by the subgraphs in the collection. There is a long-lasting conjecture on path decomposition of graphs due to Gallai, which is stated below.

\begin{conj}\label{conj}(Gallai, 1968) \cite{lovas}
If $G$ is a connected graph on $n$ vertices, then $G$ has a path decomposition of size at most $\lceil \frac{n}{2} \rceil $. 
\end{conj}
The bound mentioned in the conjecture is attained when all the vertices in a graph have odd degree since the \textbf{odd degree condition} forces each of the vertices in the graph to be an end vertex of at least one path in any path decomposition of the graph.
  That is, we need at least $\lceil \frac{n}{2} \rceil $  paths to have a path decomposition of the graph $G$. For odd graphs (graphs with each vertex having odd degree), the existence of such a path decomposition of size $\frac{n}{2}$ is discussed in \cite{lovas}. Moreover, in the same paper, it is shown that for any graph $G$ of order $n$, it is possible to have a decomposition consisting of paths and cycles satisfying the bound mentioned in the conjecture. Precisely:
 
\begin{thm} \label{LT1}
A graph of $n$ vertices can be covered by at most $\lfloor \frac{n}{2} \rfloor$ disjoint paths and circuits.
\end{thm}

 The conjecture \ref{conj} has been verified for specific families of graphs, such as those with restrictions on degree conditions \cite{lovas, LP, GF, MB, 2kreg, maxdegree6}, structural conditions \cite{botlertree, planar, girth4, seriesparallel}, etc. However, the conjecture remains open for many well-known families of graphs.

 In \cite{girth4}, the authors have shown that graphs with girth $g$ at least four have a path decomposition of size at most $\frac{o}{2}+ \lfloor (\frac{g+1}{2g}) t \rfloor$, where $t,o$ represents the number of even degree and odd degree vertices in the graph, respectively. Observe that the bound is approximately $\frac{n}{2} + \frac{t}{2g}$, which is a good approximation when $g$ is sufficiently large or $t$ is efficiently small. Validity of the Gallai's conjecture for complete bipartite graphs has been discussed in \cite{CKC}. They have shown that the path number (the minimum cardinality of a path decomposition of a graph) of the complete bipartite graph $K_{n_1,n_2}$ of order $n$  is $\frac{n}{2}$ if $n_1=n_2$ and when $n_1$ is even and $1\leq n_2 \leq n_{1}-1$ the path number is $\frac{n_1}{2}$.

The path decomposition problem on regular bipartite graphs is addressed in \cite{6RB}. They have shown that every $6$-regular bipartite graph on $n$ vertices can be decomposed into $\frac{n}{2}$ paths. 
 
 In this work, we consider another family of bipartite graphs, which is bi-regular and has girth six called Levi graphs and show that these family of bipartite graphs also satisfy the upper bound suggested by Gallai. Moreover, the upper bound of the path number that we establish in this work is strictly less than the bound in Conjecture\ref{conj} in certain cases.

Levi graph of order one, $L_{1}(m,k)$, is a bipartite graph with partition $(A,B)$, where $A$ is the collection of all subsets of $[m]=\{1,2,\dots, m\}$ of size $k-1$, and $B$ is the collection of all subsets of $[m]$ of size $k$. In this graph, a $(k-1)$-element subset (say, vertex $u$ in $A$) is adjacent to a $k$-element subset (say, vertex $v$ in $B$) if and only if $u$ is properly contained in the $k$-element subset $v$. By definition, the total number of vertices in the graph is $\binom{m}{k-1} + \binom{m}{k}$. In this discussion, we denote the total number of vertices of $L_1(m,k)$ by $n$. That is, $\binom{m}{k-1} + \binom{m}{k}=n$. Moreover, observe that the degree of each vertex in $A$ is $m-k+1$, and the degree of each vertex in $B$ is $k$. Figure \ref{fig:LG43} represents the Levi graph $L_{1}(4,3)$, where $A$ contains all the $2$-element subsets of $\{1,2,3,4\}$, and $B$ contains all the $3$-element subsets of $\{1,2,3,4\}$. The number of vertices in $A$ and $B$ is six and four, respectively. The degree of each vertex in $A$ is two, whereas the degree of each vertex in $B$ is three.\\

\definecolor{myblue}{RGB}{80,80,160}
\definecolor{mygreen}{RGB}{80,160,80}

\begin{figure}[H]
  \centering

\resizebox{3cm}{4cm}{
\begin{tikzpicture}[baseline=(current bounding box.south)]
\node[draw,black, circle] (a1) {12};
\node[draw,black, circle,below=0.1cm of a1] (a2) {13};
\node[draw,black, circle,below=0.1cm of a2] (a3) {23};
\node[draw,black, circle,below=0.1cm of a3] (a4) {14};
\node[draw,black, circle,below=0.1cm of a4] (a5) {24};
\node[draw,black, circle,below=0.1cm of a5] (a6) {34};

\node[draw,black, circle,right=4cm of a1] (b1) {123};
\node[draw,black, circle,below=0.5cm of b1] (b2) {124};
\node[draw,black, circle,below=0.5cm of b2] (b3) {134};
\node[draw,black, circle,right=4cm of a6] (b4) {234};

\node[shape=rectangle,draw=black,line width=1pt,minimum size=1.5cm,fit={(a1) (a6)}] {};
\node[shape=rectangle,draw=black,line width=1pt,minimum size=1.5cm,fit={(b1) (b4)}] {};

\node[above=0.5cm of a1,font=\color{black}\Large\bfseries] {$A$};
\node[above=0.5cm of b1,font=\color{black}\Large\bfseries] {$B$};

%\draw[-,mygreen] {-, below=of b1};
\draw[-,black] (a1) -- (b1);
\draw[-,black] (a1) -- (b2);
\draw[-,black] (a2) -- (b1);
\draw[-,black] (a2) -- (b3);
\draw[-,black] (a3) -- (b1);
\draw[-,black] (a3) -- (b4);
\draw[-,black] (a4) -- (b2);
\draw[-,black] (a4) -- (b3);
\draw[-,black] (a5) -- (b2);
\draw[-,black] (a5) -- (b4);
\draw[-,black] (a6) -- (b3);
\draw[-,black] (a6) -- (b4);

\end{tikzpicture}}
 \caption{Levi graph of order one $L_1(4,3)$}
      \label{fig:LG43}
 \end{figure}

This paper is organized as follows. Section \ref{section 1} introduces the problem, providing context and motivation for the study, as well as an overview of relevant background information on path decomposition, Gallai's path decomposition conjecture, and the Levi graph. Section \ref{section 2} presents the key lemmas used to prove the main results in this work. In Section \ref{section 3}, we prove the conjecture for the Levi graph of order one. Finally, Section \ref{section 4} discusses the path number of $L_1(m,2)$. The Gallai's conjecture guarantees a path decomposition of size $\lceil \frac{n}{2} \rceil$ for any graph of order $n$, but our upper bound of the size of the path decomposition of $L_1(m,k)$ is slightly lesser than the Gallai's bound. We show that $\lfloor \frac{n}{2} \rfloor$ paths are enough to have a path decomposition of the Levi graphs having $n$ vertices. Some authors call such graphs, a graph of order $n$ with a path decomposition of size $\lfloor \frac{n}{2} \rfloor$, as Gallai graphs\cite{trifpg}. In short, we show that the Levi graphs of order one are Gallai graphs. Additionally, we prove that the path number of $L_1(m,2)$ is $\lceil \frac{m}{2} \rceil$, which is much smaller than the proposed upper bound. The following are the main theorems in this work, and the proofs are given in sections \ref{section 3} and \ref{section 4}, respectively.\\
\textbf{Main Theorems:}
\begin{theorem} \label{m1}
    For $m \geq 2$ and $2 \leq k \leq m$, the path number of $L_1(m,k)$ having $n$ vertices is less than or equal to $\lfloor \frac{n}{2} \rfloor$. 
\end{theorem}

\begin{theorem} \label{m2}
    For $m \geq 2$, the path number of $L_1(m,2)$ is $\lfloor \frac{m}{2} \rfloor$. 
\end{theorem}

\section{Preliminaries} \label{section 2}

In this section, we discuss a few lemmas and a partition of the Levi graph, which is used in the proof of the main theorem, Theorem\ref{m1}.\\

\subsection{Partition of the Levi graph $L_1(m,k)$} 

We partition the Levi graph $L_1(m,k)$ into three different parts, namely $LLG, ULG$ and ``crossing edges". $LLG$ is an induced subgraph of $L_1(m,k)$ containing all vertices having the integer $m$. For example, figures \ref{p1} and \ref{p2} show the Levi graph $L_(4,2)$ and $LLG$ of $L_1(4,2)$, respectively. One can observe that figure \ref{p2} is obtained from figure \ref{p1} by considering the induced subgraph of $L_1(4,2)$ containing all the vertices having integer $4$. The induced subgraph on the remaining vertices of $L_1(m,k)$ is defined as $ULG$; that is, the induced subgraph $ULG$ of $L_1(m,k)$ contains all the vertices of $L_1(m,k)$ where the integer $m$ is not present. The edges of $L_1(m,k)$ that connect $LLG$ and $ULG$ are defined as ``crossing edges". Precisely, the ``crossing edges" are exactly the edges in $E(L_1(m,k)) \setminus [E(LLG) \cup E(ULG)]$. Figures \ref{p1}, \ref{p2}, \ref{p3}, and \ref{p4} show the Levi graph $L_1(4,2)$ and its partitions: $LLG$, $ULG$, and ``crossing edges". One can observe that due to the presence and absence of the integer `$m$', the induced subgraphs $LLG$ and $ULG$ are disjoint.

\begin{figure}[H]
  \centering
\subfigure[Levi graph $L_1(4,2)$]{\resizebox{3cm}{4cm}{
\begin{tikzpicture}[baseline=(current bounding box.south)]
 \node[draw,black, circle] (a1) {1};
 \node[draw,black, circle,below=0.8cm of a1] (a2) {2};
 \node[draw,black, circle,below=0.8cm of a2] (a3) {3};
 \node[draw,black, circle,below=0.8cm of a3] (a4) {4};
 \node[below=0.7cm of a4] (a5) {};

 \node[draw,black, circle,right=4cm of a1] (b1) {12};
 \node[draw,black, circle,below=0.3cm of b1] (b2) {13};
 \node[draw,black, circle,below=0.3cm of b2] (b3) {23};
 \node[draw,black, circle,below=0.3cm of b3] (b4) {14};
 \node[draw,black, circle,below=0.3cm of b4] (b5) {24};
 \node[draw,black, circle,below=0.3cm of b5] (b6) {34};

 \node[shape=rectangle,draw=black,line width=1pt,minimum size=1.5cm,fit={(a1) (a5)}] {};
 \node[shape=rectangle,draw=black,line width=1pt,minimum size=1.5cm,fit={(b1) (b6)}] {};

 \node[above=0.5cm of a1,font=\color{black}\Large\bfseries] {$A$};
 \node[above=0.5cm of b1,font=\color{black}\Large\bfseries] {$B$};

 \draw[-,Black] (a1) -- (b1);
 \draw[-,Black] (a1) -- (b2);
 \draw[-,Black] (a1) -- (b4);
 \draw[-,Black] (a2) -- (b1);
 \draw[-,Black] (a2) -- (b3);
 \draw[-,Black] (a2) -- (b5);
 \draw[-,Black] (a3) -- (b2);
 \draw[-,Black] (a3) -- (b3);
 \draw[-,Black] (a3) -- (b6);
 \draw[-,Black] (a4) -- (b4);
 \draw[-,Black] (a4) -- (b5);
 \draw[-,Black] (a4) -- (b6);
 \end{tikzpicture}}
\label{p1}}\,\,\,\,\,\,\,\,\,\,\,\,\,\,\,\,\,\,\,\,\,\,\,\,\,\,\,\,\,\,\,\,\,\,\,\,\,\,\,\,\,\,\,\,\,\,\,\,\,\,\,\,\,\,\,\,\,\,\,\,\,\,\,\,\,\,\,\,\,\,\,\,\,\,\,\,\,\,\,\,
\subfigure[$LLG$ of $L_1(4,2)$]{\resizebox{3cm}{3cm}{
 \begin{tikzpicture}[baseline=(current bounding box.south)]
 \node (a1) {};
 \node[draw,black, circle,below=0.8cm of a1] (a4) {4};
 \node[below=0.7cm of a4] (a5) {};

 \node[draw,black, circle,right=4cm of a1] (b4) {14};
 \node[draw,black, circle,below=0.3cm of b4] (b5) {24};
 \node[draw,black, circle,below=0.3cm of b5] (b6) {34};

 \node[shape=rectangle,draw=black,line width=1pt,minimum size=1.5cm,fit={(a1) (a5)}] {};
 \node[shape=rectangle,draw=black,line width=1pt,minimum size=1.5cm,fit={(b1) (b6)}] {};

 \node[above=0.5cm of a1,font=\color{black}\Large\bfseries] {$A$};
 \node[above=0.5cm of b1,font=\color{black}\Large\bfseries] {$B$};

 \draw[-,Black] (a4) -- (b4);
 \draw[-,Black] (a4) -- (b5);
 \draw[-,Black] (a4) -- (b6);
 \end{tikzpicture}}
\label{p2}}

\subfigure[$ULG$ of $L_1(4,2)$]{\resizebox{3cm}{3cm}{
\begin{tikzpicture}[baseline=(current bounding box.south)]
 \node[draw,black, circle] (a1) {1};
 \node[draw,black, circle,below=0.4cm of a1] (a2) {2};
 \node[draw,black, circle,below=0.4cm of a2] (a3) {3};

 \node[draw,black, circle,right=4cm of a1] (b1) {12};
 \node[draw,black, circle,below=0.3cm of b1] (b2) {13};
 \node[draw,black, circle,below=0.3cm of b2] (b3) {23};

 \node[shape=rectangle,draw=black,line width=1pt,minimum size=1.5cm,fit={(a1) (a3)}] {};
 \node[shape=rectangle,draw=black,line width=1pt,minimum size=1.5cm,fit={(b1) (b3)}] {};

 \node[above=0.5cm of a1,font=\color{black}\Large\bfseries] {$A$};
 \node[above=0.5cm of b1,font=\color{black}\Large\bfseries] {$B$};

 \draw[-,Black] (a1) -- (b1);
 \draw[-,Black] (a1) -- (b2);
 \draw[-,Black] (a2) -- (b1);
 \draw[-,Black] (a2) -- (b3);
 \draw[-,Black] (a3) -- (b2);
 \draw[-,Black] (a3) -- (b3);
\end{tikzpicture}}
\label{p3}}\,\,\,\,\,\,\,\,\,\,\,\,\,\,\,\,\,\,\,\,\,\,\,\,\,\,\,\,\,\,\,\,\,\,\,\,\,\,\,\,\,\,\,\,\,\,\,\,\,\,\,\,\,\,\,\,\,\,\,\,\,\,\,\,\,\,\,\,\,\,\,\,\,\,\,\,\,\,\,\,
\subfigure[``Crossing edges"]{\resizebox{3cm}{4cm}{
 \begin{tikzpicture}[baseline=(current bounding box.south)]
\node[draw, black, circle] (a1) {1};
\node[draw, black, circle, below=0.32cm of a1] (a2) {2};
\node[draw, black, circle, below=0.32cm of a2] (a3) {3};
\node[black, below=1cm of a3] (a4) {};
\node[draw, black, circle, below=0.75cm of a4] (a5) {4};
\node[black, below=0.75cm of a5] (a6) {};

\node[shape=rectangle,draw=black,line width=2pt,minimum size=1.5cm,fit={(a1) (a3)}] {};
\node[shape=rectangle,draw=black,line width=2pt,minimum size=1.5cm,fit={(a4) (a6)}] {};
\node[above=0.5cm of a1,font=\color{black}\Large\bfseries] {$A$};

\node[draw, black, circle, right=5cm of a1] (b1) {12};
\node[draw, black, circle, below=0.17cm of b1] (b2) {13};
\node[draw, black, circle, below=0.17cm of b2] (b3) {23};
\node[draw, black, circle, below=1cm of b3] (b4) {14};
\node[draw, black, circle, below=0.2cm of b4] (b5) {24};
\node[draw, black, circle, below=0.2cm of b5] (b6) {34};

\node[shape=rectangle,draw=black,line width=2pt,minimum size=1.5cm,fit={(b1) (b3)}] {};
\node[shape=rectangle,draw=black,line width=2pt,minimum size=1.5cm,fit={(b4) (b6)}] {};
\node[above=0.5cm of b1,font=\color{black}\Large\bfseries] {$B$};

\draw[-,black] (a1) -- (b4);
\draw[-,black] (a2) -- (b5);
\draw[-,black] (a3) -- (b6);

\end{tikzpicture}}
\label{p4}}

 \caption{Partition of the Levi graph $L_1(4,2)$ }
      \label{fig:Partl142}
  \end{figure}

 In general, for ease of representation, we consider a partition of $V(L_1(m,k))$ as shown in Figure \ref{fig:plg}, where $A_1$ be the set of all vertices corresponding to $(k-1)$-element subsets of $[m]$ that do not contain the integer $m$, and $A_2$ be the set of all vertices corresponding to $(k-1)$-element subsets of $[m]$ that contain the integer $m$. Similarly, $B_1$ is the set of all $k$-element subsets of $[m]$ that do not contain the integer $m$, and $B_2$ be the remaining vertices in $L_1(m,k)$. Observe that $(A_1, B_1)$ forms the vertex set of the bipartite graph $ULG$ of $L_1(m,k)$, and $(A_2, B_2)$ that of the bipartite graph $LLG$ of $L_1(m,k)$. 
 Note that the presence of the integer `$m$' in each of the vertices of $A_2$ and absence of the integer `$m$' in the vertices of $B_1$ prevents any ``crossing edges" between $A_2$ and $B_1$. We make the following observation regarding ``crossing edges" of $L_1(m,k)$.

\begin{observation} \label{obcr}
    The ``crossing edges" form a perfect matching between $A_1$ and $B_2$. That is, 
    \begin{itemize}
        \item A ``crossing edge" has one end point in $A_1$ and the other in $B_2$,
        \item No two ``crossing edges" share an end vertex,
        \item There are $|A_1| = |B_2|$ ``crossing edges" in $L_1(m,k)$.
    \end{itemize}
     
\end{observation}

\begin{proof}
    Since $A_1$ contains the set of all $(k-1)$-element subsets of $[m]$ without the integer $m$ and $B_2$ contains the set of all $k$-element subsets of $[m]$ with the integer $m$, we observe that $|A_1|=|B_2|$. Moreover, since each vertex in $B_2$ is obtained by adding $m$ to a vertex in $A_1$, there is an edge connecting the respective vertices in $A_1$ and $B_2$. Similarly for a vertex $v$ (which is a $k$-element subset) in $B_2$, there exists a unique adjacent vertex $u$ (a $(k-1)$-element subset that does not contain $m$) in $A_1$. Therefore, no ``crossing edges" share a common vertex. 
    
    \end{proof}

\begin{figure}[H]
  \centering
\resizebox{3cm}{4cm}{
\begin{tikzpicture}[baseline=(current bounding box.south)]

\node (a1) {$A_1$};
\node[below=0.6cm of a1] (a2) {};
\node[below=1cm of a2] (a3) {$A_2$};
\node[below=of a3] (a4) {};

\node[right=4cm of a1] (b1) {$B_1$};
\node[below=0.6cm of b1] (b2) {};
\node[below=1cm of b2] (b3) {$B_2$};
\node[below=of b3] (b4) {};

\node[shape=rectangle,draw=black,line width=1pt,fit={(a1) (a4)},text width=1.75cm, text height=5cm] {};
\node[shape=rectangle,draw=black,line width=1pt,fit={(a3) (a4)},text width=1.75cm, text height=2.6cm] {};
\node[shape=rectangle,draw=black,line width=1pt,fit={(b1) (b4)},text width=1.75cm, text height=5cm] {};
\node[shape=rectangle,draw=black,line width=1pt,fit={(b3) (b4)},text width=1.75cm, text height=2.6cm] {}; %this is for the extra rectangle in the main rectangle i.e. to create the line in the rectangle

\node[above=1cm of a1,font=\color{black}\Large\bfseries] {$A$};
\node[above=1cm of b1,font=\color{black}\Large\bfseries] {$B$};

\end{tikzpicture}}
 \caption{Partition of Levi graph}
      \label{fig:plg}
  \end{figure}

\subsection{Lemmas}

The following lemma shows that $LLG$ of a Levi graph $L_1(m,k)$ is itself a Levi graph. 

\begin{lemma}\label{llg}
For $m \geq 3$ and $k \geq 3$, $LLG$ of $L_{1}(m,k) \cong L_{1}(m-1,k-1)$.    
   
\end{lemma}

\begin{proof}

The isomorphism between $LLG$ of $L_1(m,k)$ and $L_1(m-1,k-1)$ is obtained by removing the integer `$m$' from each vertex of $LLG$ of $L_1(m,k)$. Precisely, consider a mapping $f: V(LLG) \longrightarrow V(L_1(m-1,k-1))$ such that $f(\{a_1, a_2, \dots, a_t, m\}) = \{a_1, a_2, \dots, a_t\}$. If $t$ is $k-1$, then $f$ maps a vertex, which is a $k$-element subset of $[m]$ containing $m$ in $LLG$ to the corresponding $(k-1)$-element subset of $[m]$ which is in $L_1(m-1,k-1)$. Similarly, if $t=k-2$, then the $(k-1)$-element subset of $LLG$ is mapped to the corresponding $(k-2)$-element subset in $L_1(m-1,k-1)$. One can easily observe that this mapping is an isomorphism between $LLG$ of $L_1(m,k)$ and $L_1(m-1,k-1)$. For completeness of the proof we have added the arguments to support the isomorphism but one can skip it, if it is clear. \\

\noindent $f$ is one-one: 
For $t_1, t_2 \in \{k-1, k-2\}$, let $f(x_1)= \{a_1, a_2, \dots , a_{t_1}\}$ and $f(x_2)= \{b_1, b_2, \dots , b_{t_2}\}$ for the vertices $x_1=\{a_1, a_2, \dots , a_{t_1}, m\}$ and $x_2=\{b_1, b_2, \dots , b_{t_2}, m\}$ of $LLG$ of $L_1(m,k)$. Then, if $f(x_1)=f(x_2)$, then $\{a_1, a_2, \dots , a_{t_1}\} = \{b_1, b_2, \dots , b_{t_2}\}$ which implies $\{a_1, a_2, \dots , a_{t_1}, m\} = \{b_1, b_2, \dots , b_{t_2}, m\}$ which implies $x_1=x_2$.\\

\noindent $f$ is onto:
For $t_1=\{k-1, k-2\}$, let $y=\{a_1, a_2, \dots , a_{t_1}\}$ be a vertex of $L_1(m-1,k-1)$. Then $\{a_1, a_2, \dots , a_{t_1}, m\}$ is a vertex, say $x$, in $LLG$ of $L_1(m,k)$ such that $f(x)=y$.\\

\noindent $f$ preserves adjacency:
For $t_1, t_2 \in \{k-1, k-2\}$, let $x=\{a_1, a_2, \dots , a_{t_1}, m\}$ and $y=\{b_1, b_2, \dots , b_{t_2}, m\}$. Without loss of generality $\{a_1, a_2, \dots , a_{t_1}, m\} \subseteq \{b_1, b_2, \dots , b_{t_2}, m\}$. That is $\{a_1, a_2, \dots , a_{t_1}\} \subseteq \{b_1, b_2, \dots , b_{t_2}\}$, which implies $f(x) \subseteq f(y)$. This implies that $f(x)$ and $f(y)$ are adjacent.

\end{proof}

By construction itself, one may observe that $ULG$ of $L_1(m,k)$ is isomorphic to $L_1(m-1,k)$.

\begin{lemma} \label{ulg}
     $ULG$ of $L_1(m,k)$ is the Levi graph $L_1(m-1,k)$, for $m\geq 3$, $k\geq 2$.
     \qed
     
\end{lemma}

Now we discuss an inequality which is used to prove Theorem \ref{m1} and in the later discussions. Observe that the inequality, 
\begin{equation}\label{ab}  
 \left \lfloor \frac{a}{2} \right\rfloor +  \left\lfloor \frac{b}{2} \right\rfloor \leq  \left \lfloor \frac{a+b}{2} \right \rfloor
 \end{equation} 
 is valid for all natural numbers $a$ and $b$. 
 This is because when both $a$ and $b$ are odd numbers, since $\left \lfloor \frac{a}{2} \right \rfloor  = \frac{a-1}{2} $ and $\left \lfloor \frac{b}{2} \right \rfloor  = \frac{b-1}{2} $, we have
$$ \left \lfloor \frac{a}{2} \right\rfloor +  \left\lfloor \frac{b}{2} \right\rfloor =\frac{a+b-2}{2} \leq \frac{a+b}{2}.$$ 

Similarly, when $a$ and $b$ are even numbers, then $$ \left \lfloor \frac{a}{2} \right\rfloor +  \left\lfloor \frac{b}{2} \right\rfloor =\frac{a+b}{2} \leq \frac{a+b}{2}.$$ 

If either $a$ or $b$ is odd, then also it is straight forward to see that $$ \left \lfloor \frac{a}{2} \right\rfloor +  \left\lfloor \frac{b}{2} \right\rfloor \leq \frac{a+b}{2}.$$  

Therefore, by choosing $a=\binom{m+1}{k}$ and $b=\binom{m+1}{k-1}$ in the \eqref{ab} and using Pascal's formula \cite{brualdi}, 
 \begin{equation*}
 \binom{m+1}{k} = \binom{m}{k} + \binom{m}{k-1}
  \end{equation*} we have Lemma \ref{pascal}.

\begin{lemma}\label{pascal}
For all $m, k \in \mathbb{N}$, with $1\leq k \leq m$ we have    
\begin{equation}\label{pascalst}
\left\lfloor \frac{\binom{m}{k-1} + \binom{m}{k}}{2} \right\rfloor + \left\lfloor \frac{\binom{m}{k-2} + \binom{m}{k-1}}{2} \right\rfloor \leq \left\lfloor \frac{\binom{m+1}{k-1} + \binom{m+1}{k}}{2} \right\rfloor.   
\end{equation}

\end{lemma}

The following theorems discuss `border' cases $k=m$ and $k=2$ of Theorem \ref{m1}.

\begin{theorem}\label{star} 
For $m\geq 2$ and $k=m$, the path number of $L_1(m,k)$ on $n=m+1$ vertices is at most $\lfloor \frac{m+1}{2} \rfloor$.
\end{theorem}

\begin{proof}
  
  Observe that when $k=m$, the right side $B$ (see Figure \ref{l1mm}) of the bipartite graph $L_1(m,m)$ contains a unique vertex (since $\binom{m}{m}=1$) representing the set $\{1,2,\dots,m\}$ and the other side $A$ contains all $(m-1)$ subsets of $\{1,2,\dots,m\}$. Hence $L_1(m,m)$ represents the star graph $K_{1,m}$. Now to discuss the path decomposition of this graph, define $x_1, x_2, \dots , x_m$ as the vertices in $A$ and $y=\{1,2,\dots,m\}$ as the vertex in $B$ of $L_1(m,m)$.

\begin{figure}[H]
  \centering
\resizebox{3cm}{4cm}{
\begin{tikzpicture}[baseline=(current bounding box.south)]

\node[draw,black, circle] (a1) {$x_{1}$};
\node[draw,black, circle,below=0.3cm of a1] (a2) {$x_{2}$};
\node[below=0.3cm of a2] (a3) {.};
\node[below=0.3cm of a3] (a4) {.};
\node[below=0.3cm of a4] (a5) {.};
\node[draw,black, circle,below=0.3cm of a5] (a6) {$x_{n}$};

\node[ right=4cm of a1] (b1) {};
\node[ below=0.7cm of b1] (b2) {};
\node[draw,black, circle,below=0.7cm of b2] (b3) {$y$};
\node[ below=0.7cm of b3] (b4) {};
\node[ below=0.7cm of b4] (b5) {};

\node[shape=rectangle,draw=black,line width=1pt,minimum size=1cm,fit={(a1) (a6)}] {};
\node[shape=rectangle,draw=black,line width=1pt,minimum size=1cm,fit={(b1) (b5)}] {};

\node[above=0.5cm of a1,font=\color{black}\Large\bfseries] {$A$};
\node[above=0.5cm of b1,font=\color{black}\Large\bfseries] {$B$};

%\draw[-,mygreen] {-, below=of b1};
\draw[-,red] (a1) -- (b3);
\draw[-,red] (a2) -- (b3);
\draw[-,blue] (a3) -- (b3);
\draw[-,green] (a6) -- (b3);

\end{tikzpicture}}
 \caption{Levi graph $L_1(m,m)$}
      \label{l1mm}
  \end{figure}

 Consider the collection of paths $\mathcal{P} = \{P_1, P_2, \dots P_t\}$ in the graph $L_{1}(m,m)$ where $P_i:x_{2i-1}, y,  x_{2i}$ is the path. For example, $P_1: x_1, y, x_3$; $P_2: x_3, y, x_4$; etc. When $m$ is even, the decomposition requires $t=\frac{m}{2}$ paths of the form $P_i$ as described above but when $m$ is odd, the first $t=\frac{m-1}{2}$ paths follow the $P_i$ structure, and the final path is the single edge $(x_m, y)$. This results in a total of $\frac{m+1}{2}$ paths. Hence, in both cases, that is, when $m$ is even and $m$ is odd, we have a path decomposition of size $\lceil \frac{m}{2} \rceil $. Observe that since $L_{1}(m,m)$ contains $m+1$ vertices, this proof implies that the path number of $L_{1}(m,m)$ is at most $\lfloor \frac{m+1}{2} \rfloor$.

\end{proof}

 Now, in the other extreme case, that is, when $k=2$, we show that the path number of $L_1(m,k)$ on $n$ vertices is bounded above by $\lfloor \frac{n}{2}\rfloor $.
\begin{theorem}\label{t1}
    For all $m \ge 2$, the path number of $L_{1}(m,2)$ on $n$ vertices is at most $\lfloor \frac{n}{2}\rfloor$.
\end{theorem}

\begin{proof}
    We prove this by induction on $m$.\\ 
    \textbf{Base Cases}
    \begin{enumerate}
        \item Case $m=2$: The graph $L_1(2,2)$ is a star graph, and the result follows directly from Theorem \ref{star}.
        \item Case $m=3$: The graph $L_1(3,2)$ is shown in figure \ref{l32} and the paths $P_1: \{1,3\}, 1, \{1,2\}, 2, \{2,3\}$\ and $P_2: \{1,3\}, 3, \{2,3\}$ decompose the edge set into paths.
    \end{enumerate} 
    \textbf{Inductive Step ($m=4$)}:\\
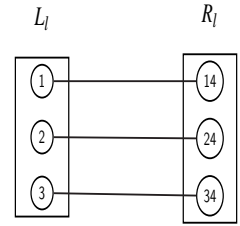
  Consider $L_1(4,2)$ (Figure \ref{l42}). Observe that $LLG$ of $L_1(4,2)$ is a star graph which contains the vertices $4$, $\{1,4\}$, $\{2,4\}$ and $\{3,4\}$ with the edges $(4, \{1,4\})$, $(4, \{2,4\})$ and $(4, \{3,4\})$. We use the path $P': \{1,4\}, 4, \{2,4\}$ and $P'': 4, \{3,4\}$ to decompose $LLG$ of $L_1(4,2)$. By lemma \ref{ulg}, the $ULG$ of $L_1(4,2)$ is isomorphic to $L_1(3,2)$ and the above paragraph in this proof discusses a path decomposition supporting result. The remaining edges in $L_1(4,2)$ are the ``crossing edges" (see Introduction). Before dealing with the ``crossing edges", let us understand the present status of the path decomposition that we have made for $L_1(4,2)$. There are $n=10$ vertices in $L_1(4,2)$ and we have made four paths in total to decompose the edge sets belonging to $ULG$ and $LLG$ of $L_1(4,2)$. Since Galli's conjecture allows at most $\frac{n}{2}=5$ paths, we are left with at most one path to manage the crossing edges. In the following discussion, we show that the ``crossing edges" can be attached to the existing paths that we have already made, and this shows that the path number of $L_1(4,2)$ is at most $4$. From figure \ref{l42}, observe that $(1,\{1,4\})$, $(2,\{2,4\})$, $(3,\{3,4\})$ are the ``crossing edges" of $L_1(4,2)$. Note that if we consider $L_1(4,2)$ without the ``crossing edges", the vertices $\{1,4\}$, $\{2,4\}$, and $\{3,4\}$ are odd degree vertices and they are end vertices of some path in any path decomposition of $L_1(4,2)$ minus crossing edges. In our case, $\{1,4\}$ and $\{2,4\}$ are the end point of the path $P': \{1,4\}, 4, \{2,4\}$ and $\{3,4\}$ is an end point of the path $P'': 4, \{3,4\}$ in the path decomposition we have made. We add the ``crossing edges" $(1,\{1,4\})$ and $(2,\{2,4\})$ to the path $P'$ and $(3,\{3,4\})$ to the path $P''$, resulting a path decomposition of $L_1(4,2)$ in the following way.\\
    $P_1: \{1,3\}, 1, \{1,2\}, 2, \{2,3\}$\\
    $P_2: \{1,3\}, 3, \{2,3\}$\\
    $P_3: 1, \{1,4\}, 4, \{2,4\}, 2$\\
    $P_4: 4, \{3,4\}, 3$\\
    which proves the statement in this case.\\

    \textbf{Note:} The proof of the case $L_1(4,2)$ can be written just by writing the path $P_1, P_2, P_3, P_4$ directly, but we have gone through a deeper discussion to explain the structure of the graph and the idea behind the path decomposition that is used in the general case.

\begin{minipage}[b]{0.45\textwidth}
\begin{figure}[H]
  \centering

\resizebox{3cm}{4cm}{
\begin{tikzpicture}[baseline=(current bounding box.south)]

\node[draw,black, circle] (a1) {1};
\node[draw,black, circle,below=0.8cm of a1] (a2) {2};
\node[draw,black, circle,below=0.8cm of a2] (a3) {3};
\node[draw,black, circle,below=0.8cm of a3] (a4) {4};
\node[below=0.7cm of a4] (a5) {};

\node[draw,black, circle,right=4cm of a1] (b1) {12};
\node[draw,black, circle,below=0.3cm of b1] (b2) {13};
\node[draw,black, circle,below=0.3cm of b2] (b3) {23};
\node[draw,black, circle,below=0.3cm of b3] (b4) {14};
\node[draw,black, circle,below=0.3cm of b4] (b5) {24};
\node[draw,black, circle,below=0.3cm of b5] (b6) {34};

\node[shape=rectangle,draw=black,line width=1pt,minimum size=1.5cm,fit={(a1) (a5)}] {};
\node[shape=rectangle,draw=black,line width=1pt,minimum size=1.5cm,fit={(b1) (b6)}] {};

\node[above=0.5cm of a1,font=\color{black}\Large\bfseries] {$A$};
\node[above=0.5cm of b1,font=\color{black}\Large\bfseries] {$B$};

%\draw[-,mygreen] {-, below=of b1};
\draw[-,red] (a1) -- (b1);
\draw[-,red] (a1) -- (b2);
\draw[-,blue] (a1) -- (b4);
\draw[-,red] (a2) -- (b1);
\draw[-,red] (a2) -- (b3);
\draw[-,blue] (a2) -- (b5);
\draw[-,green] (a3) -- (b2);
\draw[-,green] (a3) -- (b3);
\draw[-,brown] (a3) -- (b6);
\draw[-,blue] (a4) -- (b4);
\draw[-,blue] (a4) -- (b5);
\draw[-,brown] (a4) -- (b6);

\end{tikzpicture}}					
\caption{$L_1(4,2)$}
      \label{l42}	
	
\end{figure}
\end{minipage}
  \hfill
  \begin{minipage}[b]{0.45\textwidth}
\begin{figure}[H]
  \centering

  \resizebox{3cm}{3cm}{
\begin{tikzpicture}[baseline=(current bounding box.south)]

\node[draw,black, circle] (a1) {1};
\node[draw,black, circle,below=0.5cm of a1] (a2) {2};
\node[draw,black, circle,below=0.5cm of a2] (a3) {3};

\node[draw,black, circle,right=4cm of a1] (b1) {12};
\node[draw,black, circle,below=0.3cm of b1] (b2) {13};
\node[draw,black, circle,below=0.3cm of b2] (b3) {23};

\node[shape=rectangle,draw=black,line width=1pt,minimum size=1.5cm,fit={(a1) (a3)}] {};
\node[shape=rectangle,draw=black,line width=1pt,minimum size=1.5cm,fit={(b1) (b3)}] {};

\node[above=0.5cm of a1,font=\color{black}\Large\bfseries] {$A$};
\node[above=0.5cm of b1,font=\color{black}\Large\bfseries] {$B$};

%\draw[-,mygreen] {-, below=of b1};
\draw[-,red] (a1) -- (b1);
\draw[-,red] (a1) -- (b2);
\draw[-,red] (a2) -- (b1);
\draw[-,red] (a2) -- (b3);
\draw[-,green] (a3) -- (b2);
\draw[-,green] (a3) -- (b3);

\end{tikzpicture}}	
   \caption{$L_1(3,2)$}
      \label{l32}
	\end{figure}
    \end{minipage}
    
\noindent \textbf{Induction hypothesis:} The statement of Theorem \ref{t1} is valid for all $L_1(t,2)$, where $t < m$.\\
\textbf{Proof of $L_1(m,2)$:} Here also we use an argument similar to the base case. Observe that $LLG$ of $L_1(m,2)$ is a star graph with center vertex $m$. In this case, Theorem \ref{star} gives a path decomposition $\mathcal{P}_1$ of size at most $\lfloor \frac{m}{2} \rfloor$. According to Lemma \ref{ulg}, $ULG$ of $L_1(m,2)$ is isomorphic to $L_1(m-1,2)$ and by the induction hypothesis there exists a path decomposition $\mathcal{P}_2$ of size at most $\left\lfloor \frac{\binom{m-1}{1} + \binom{m-1}{2}}{2} \right\rfloor$. Observe that the edge set of $LLG$ of $L_1(m,2)$ and that of $ULG$ of $L_1(m,2)$ are disjoint. Hence, we have made $\lfloor \frac{m}{2} \rfloor+\left\lfloor \frac{\binom{m-1}{1} + \binom{m-1}{2}}{2} \right\rfloor$ paths to include the edges of $LLG$ and $ULG$ of $L_1(m,2)$. By lemma \ref{pascal} and Pascal's formula, we know that $\left\lfloor \frac{m}{2} \right\rfloor + \left\lfloor \frac{\binom{m-1}{1} + \binom{m-1}{2}}{2} \right\rfloor \leq \left\lfloor \frac{\binom{m}{1} + \binom{m}{2}}{2} \right\rfloor$ which is the upper bound mentioned in the statement of Theorem \ref{t1} in the case of $L_1(m,2)$. That is, we have to add the remaining ``crossing edges" without increasing the number of paths. Again, we make use of an argument similar to the base case.

Observe that being the pendent vertices of a star graph, each vertex in $B_2$ (see figure \ref{fig:plg}) has degree one, and it is an end vertex of some path in $\mathcal{P}_2$. Since we have $|B_2|$ ``crossing edges" (see observation \ref{obcr}), we attach ``crossing edges" to the paths having end vertex in $B_2$ and denote this new set of paths as $\mathcal{P'}_2$. Then $\mathcal{P}_1 \cup \mathcal{P'}_2$ is the path decomposition of $L_1(m,2)$ of the desired size.

\end{proof}

\section{Proof of Theorem 1.1} \label{section 3}
 For $m=2$, the integer $k$ should be $2$ and Theorem \ref{star} proves the main result in this case. For $m=3$, the integer $k$ can take value $2$ or $3$. Theorem \ref{t1} and Theorem \ref{star} prove the result for $L_1(3,2)$ and $L_1(3,3)$, respectively. For $k=2$ and $k=m$ the statement of Theorem \ref{m1} is valid by Theorem \ref{t1} and Theorem \ref{star}. The following theorem proves the rest of the cases.

 \begin{theorem} \label{pm1}
    For $m \geq 4$ and $3 \leq k \leq m-1$, the path number of $L_1(m,k)$ on $n$ vertices is at most $\lfloor \frac{n}{2} \rfloor$. 
\end{theorem}

\begin{proof}
We prove the result by using induction on $m$.\\

\noindent\textbf{Base case:} When $m=4$, the value of $k$ can be $2,3$ or $4$. The cases $k=2$ and $k=4$ follows from Theorem \ref{t1} and Theorem \ref{star}, respectively. When $k=3$, figure \ref{l143} shows a path decomposition of $L_1(4,3)$ having size $\left\lfloor \frac{\binom{4}{2} + \binom{4}{3}}{2} \right\rfloor = 5$.

\begin{figure}[H]
  \centering

  \resizebox{3cm}{4cm}{
  \begin{tikzpicture}[baseline=(current bounding box.south)]
\node[draw,black, circle] (a1) {12};
\node[draw,black, circle,below=0.1cm of a1] (a2) {13};
\node[draw,black, circle,below=0.1cm of a2] (a3) {23};
\node[draw,black, circle,below=0.1cm of a3] (a4) {14};
\node[draw,black, circle,below=0.1cm of a4] (a5) {24};
\node[draw,black, circle,below=0.1cm of a5] (a6) {34};

\node[draw,black, circle,right=4cm of a1] (b1) {123};
\node[draw,black, circle,below=0.5cm of b1] (b2) {124};
\node[draw,black, circle,below=0.5cm of b2] (b3) {134};
\node[draw,black, circle,right=4cm of a6] (b4) {234};

\node[shape=rectangle,draw=black,line width=1pt,minimum size=1.5cm,fit={(a1) (a6)}] {};
\node[shape=rectangle,draw=black,line width=1pt,minimum size=1.5cm,fit={(b1) (b4)}] {};

\node[above=0.5cm of a1,font=\color{black}\Large\bfseries] {$A$};
\node[above=0.5cm of b1,font=\color{black}\Large\bfseries] {$B$};

%\draw[-,mygreen] {-, below=of b1};
\draw[-,red] (a1) -- (b1);
\draw[-,red] (a1) -- (b2);
\draw[-,red] (a2) -- (b1);
\draw[-,red] (a2) -- (b3);
\draw[-,green] (a3) -- (b1);
\draw[-,green] (a3) -- (b4);
\draw[-,yellow] (a4) -- (b2);
\draw[-,yellow] (a4) -- (b3);
\draw[-,yellow] (a5) -- (b2);
\draw[-,yellow] (a5) -- (b4);
\draw[-,brown] (a6) -- (b3);
\draw[-,brown] (a6) -- (b4);

\end{tikzpicture}}
\caption{Path decomposition of $L_1(4,3)$}
      \label{l143}

\end{figure}

\noindent\textbf{Induction hypothesis:} The statement of Theorem \ref{pm1} is valid for all $L_1(t,k)$, where $t<m$ and $3\leq k \leq m-1$.\\
\textbf{Proof of $L_1(m,k)$:}\\
The proof of this theorem is similar to that of Theorem \ref{t1}. Consider the following three cases.\\
\textbf{Case 1:} $k$ is odd and $m$ is odd,\\
\textbf{Case 2:} $k$ is odd and $m$ is even,\\
\textbf{Case 3:} $k$ is even.\\

\noindent\textit{Proof of \textbf{Case 1:}} Since each $(k-1)$-element subset in $L_1(m,k)$ has degree $m-k+1$ and the $k$-element subset in $L_1(m,k)$ has degree $k$, the \textbf{case 1} implies that each vertex in $L_1(m,k)$ has odd degree. In this case, \cite{lovas} proves the result.\\

\noindent\textit{Proof of \textbf{Case 2:}} When $k$ is odd and $m$ is even, observe that $m-k+1$ is even. In this case we use a similar technique in the proof of Theorem \ref{t1}. By Lemma \ref{ulg}, $ULG$ of $L_1(m,k)$ is isomorphic to $L_1(m-1,k)$ and according to the induction hypothesis, there exists a path decomposition $\mathcal{P}_2$ of size $\left\lfloor \frac{\binom{m-1}{k-1} + \binom{m-1}{k}}{2} \right\rfloor$. By Lemma \ref{llg}, $LLG$ of $L_1(m,k)$ is isomorphic to $L_1(m-1,k-1)$ and by the induction hypothesis, there exists a path decomposition $\mathcal{P}_1$ of size $\left\lfloor \frac{\binom{m-1}{k-2} + \binom{m-1}{k-1}}{2} \right\rfloor$. Now to handle the ``crossing edges", we use the odd degree property of vertices in $A_1$ of $L_1(m,k)$ (see Figure \ref{fig:plg}). Note that each vertex in $A_1$ has degree $m-k$, which is odd in this case. Then, by odd degree condition, each vertex of $A_1$ is an end vertex of some paths in any path decomposition of $ULG$ of $L_1(m,k)$. Since the number of ``crossing edges" is $|A_1|$ (by Observation \ref{obcr}) we attach each ``crossing edge" to a vertex of $A_1$ in a path where it is the end vertex. This modifies the path decomposition $\mathcal{P}_2$ of $ULG$ to a path decomposition $\mathcal{P'}_2$, containing the `crossing edges" where $|\mathcal{P'}_2| = |\mathcal{P}_2|$. Finally, $\mathcal{P} = \mathcal{P}_1 \cup \mathcal{P'}_2$ is a path decomposition of $L_1(m,k)$ and 
\begin{equation*}
    |\mathcal{P}| \leq \left\lfloor \frac{\binom{m-1}{k-1} + \binom{m-1}{k}}{2} \right\rfloor + \left\lfloor \frac{\binom{m-1}{k-2} + \binom{m-1}{k-1}}{2} \right\rfloor \leq \left\lfloor \frac{\binom{m}{k-1} + \binom{m}{k}}{2} \right\rfloor (\text{by Lemma \ref{pascal}}).
\end{equation*}
This proves Theorem \ref{pm1} in this case.\\

\noindent\textit{Proof of \textbf{Case 3:}} This is the case where $k$ is even. The first part of the proof is similar to the proof of \textbf{ case 2}. By Lemma \ref{ulg}, $ULG$ of $L_1(m,k)$ is isomorphic to $L_1(m-1,k)$ and according to the induction hypothesis, there exists a path decomposition $\mathcal{P}_2$ of size $\left\lfloor \frac{\binom{m-1}{k-1} + \binom{m-1}{k}}{2} \right\rfloor$. By Lemma \ref{llg}, $LLG$ of $L_1(m,k)$ is isomorphic to $L_1(m-1,k-1)$ and according to the induction hypothesis, there exists a path decomposition $\mathcal{P}_1$ of size $\left\lfloor \frac{\binom{m-1}{k-2} + \binom{m-1}{k-1}}{2} \right\rfloor$.

But in this case, to handle the ``crossing edges", we use the odd degree property of vertices in $B_2$ of $L_1(m,k)$ (see Figure \ref{fig:plg}). Note that each vertex in $B_2$ has degree $k-1$ in $LLG$ of $L_1(m,k)$ which is odd in this case. Then by odd degree condition, each vertex of $B_2$ is an end vertex of some paths in any path decomposition of $LLG$ of $L_1(m,k)$. Since the number of ``crossing edges" is $|B_2|$ (by Observation \ref{obcr}) we attach each crossing edge to a vertex of $B_2$ in a path where it is the end degree vertex. This modifies the path decomposition $\mathcal{P}_1$ of $ULG$ to a path decomposition $\mathcal{P'}_1$, containing the crossing edges where $|\mathcal{P'}_1| = |\mathcal{P}_1|$. Finally, $\mathcal{P} = \mathcal{P'}_1 \cup \mathcal{P}_2$ is a path decomposition of $L_1(m,k)$ and 
\begin{equation*}
    |\mathcal{P}| \leq \left\lfloor \frac{\binom{m-1}{k-1} + \binom{m-1}{k}}{2} \right\rfloor + \left\lfloor \frac{\binom{m-1}{k-2} + \binom{m-1}{k-1}}{2} \right\rfloor \leq \left\lfloor \frac{\binom{m}{k-1} + \binom{m}{k}}{2} \right\rfloor (\text{by Lemma \ref{pascal}}).
\end{equation*}
This proves Theorem \ref{pm1} in this case.\\

\end{proof}

\section{Minimum path decomposition} \label{section 4}

In this section, we discuss a minimum path decomposition of Levi graph $L_1(m,2)$. Some authors have used the phrase minimal path decomposition \cite{CKC} instead of the phrase minimum path decomposition. Let us start with formal definition of minimum path decomposition of a graph.

\begin{definition}(Minimum path decomposition):

A minimum path decomposition $\mathcal{P}$ of a graph $G$ is a path decomposition (see Introduction) of $G$ such that for any other path decomposition $\mathcal{P'}$ of $G$ we have $|\mathcal{P}| \leq |\mathcal{P'}|$. The number of paths in a minimum path decomposition of a graph $G$ is known as the path number of $G$ and is denoted by $pn(G)$.

\end{definition}

Many researchers have studied the path number of several families of graphs \cite{CKC, path num}. 
     The path number of a tree is known to be half the number of odd degree vertices in it, and the path number of a cubic graph is half the number of its vertices \cite{path num}.
     The path number of the complete graph $K_n$ is discussed by several researchers \cite{path num, Walecki} and is known to be $\lceil \frac{n}{2} \rceil$.

    In general, identifying a minimum path decomposition of a graph is challenging because one must prove that the obtained number of paths is indeed the minimum possible. A common strategy for establishing lower bounds on the path number involves analyzing the number of vertices in the graph having odd degree or showing that the paths used in the decomposition are of maximum possible length.

    To determine the path number of the Levi graph $L_1(m,2)$, for $m\geq 2$, we use an auxiliary graph isomorphic to the complete graph $K_m$. Since the minimum path decomposition of $K_m$ is already known\cite{Walecki}, this allows us to use the existing results. In the following subsection, we describe a minimum path decomposition of $K_m$, as given in the proof of Proposition 3.1 in \cite{Walecki}. The author describes a path decomposition of the complete graph $K_m$ based on the parity of $m$. When $m$ is even, the decomposition consists of exactly $\frac{m}{2}$ Hamiltonian paths, each of which covers all m vertices. In contrast, when m is odd, the decomposition consists of $\lfloor \frac{m}{2} \rfloor$ Hamiltonian paths, along with an additional path of length $\frac{m-1}{2}$. We adapt ideas from \cite{Walecki} and present a detailed algebraic proof for $pn(K_m)=\lceil \frac{m}{2} \rceil $, which will be useful in our analysis of path number of $L_1(m,2)$.

\subsection{Minimum path decomposition of complete graph $K_m$} \label{subsec4.1}

    Let $\{1,2,\dots,m\}$ be the vertex set of the complete graph $K_m$. By the expression$\pmod m$, we mean the modulo operation with respect to the integer $m$, where $t\pmod m$ represents the reminder of $\frac{t}{m}$. In this discussion, we take $m \pmod m = 0 \pmod m =m$. Consider the following cases: \\
\noindent\textbf{Case 1 ($m$ is even):} 

    Let $\mathcal{P}= \{P_1,P_2,\dots, P_\frac{m}{2}\}$, where each $P_i$ is a sequence of $m$ vertices, $P_i: i, i+1, i-1\pmod m,\dots , i+k \pmod m, i-k+m \pmod m, \dots, i+\frac{m}{2}+1 \pmod m, i+\frac{m}{2} \pmod m$ together with edges between consecutive pairs of vertices in the sequence. In the above sequence, $ i+k \pmod m$ and $i-k+m \pmod m$ are vertices in the positions $2k^{\text{th}}$ and $(2k+1)^{\text{th}}$ respectively, where $0 \le k \le \frac{m}{2}$.  
    
    Precisely, for $0 \le k \le m/2$, a vertex $v$ in the path $P_i$ is defined as:

\begin{equation} 
    v=
    \begin{cases}
        i+k \;\;\;\;\;\;\;\ \pmod m & \text{if } v \text{ is in } 2k^{th}\text{ position}\\
        i-k+m \pmod m & \text{if } v \text{ is in } (2k+1)^{th} \text{ position}.
    \end{cases}
    \label{v}
\end{equation}

We understand the positioning of the vertices in $P_i$ as follows: Each $P_i$ begins with the vertex $i$ in the first position. Then, the vertices $i+1, i+2, \dots i+\frac{m}{2}$ are placed in even positions $2, 4, \dots m $, respectively. Now, we place $i + \frac{m}{2} +1 \pmod m , i+\frac{m}{2} +2 \pmod m, \dots i+\frac{m}{2} + \frac{m}{2}-1 \pmod m$ in the remaining odd positions of $P_i$,  starting from position $(m-1)$ down to $3$, in descending order. 
For example, $P_2$ in $L_1(6,2)$ is $P_2: 2, 3, 1, 4, 6, 5$.

Observe that since $1 \le i \le m/2$ and all additions are taken with modulo $m$, no two vertices in any given $P_i$ receive the same value. Therefore, all $m$ vertices in $P_i$ are distinct. This implies that $P_i$ is a valid path in $K_m$, and, in particular, it forms a Hamiltonian path in $K_m$.

Furthermore, observe that each Hamiltonian path in $K_m$ consists of exactly $m - 1$ edges. With a total of $m/2$ such edge-disjoint Hamiltonian paths, the combined edge count is:
\[
\frac{m}{2} \cdot (m - 1) = \frac{(m - 1) \cdot m}{2},
\]
which equals the total number of edges in the complete graph $K_m$.

Since no path in a $m$-vertex graph can contain more than $m - 1$ edges, a Hamiltonian path decomposition is indeed a minimum path decomposition of $K_m$.

In summary, if we can prove that the paths $P_i$  for $i \in [m/2]$ are edge-disjoint, we obtain a minimum path decomposition of $K_m$. The following claim addresses this remaining part.

\begin{claim} \label{pipj}
 For $i\neq j$, paths $P_i$ and $P_j$ are edge disjoint.
\end{claim}

\begin{proof}
Consider an edge $(u,v) \in E(P_i)$, then from \eqref{v} it has one of the two following forms depending on the parity of the position of the vertex $u$.
\begin{equation} \label{epi}
    (u,v)=
    \begin{cases}
        \Big(i+k \pmod m,  i-k+m  \pmod m\Big)  &\text{if } u \text{ is in } 2k^{th}\text{ position}\\
        \text{or} \\
        \Big(i-k+m \pmod m, i+k+1  \pmod m \Big) &\text{if } u \text{ is in } (2k+1)^{\text{th}}\text{ position.}
        
    \end{cases}  
\end{equation}

Now, on contrary,  suppose that there are two paths, say $P_i$ and $P_j$, in the path decomposition $\mathcal{P}$  sharing a common edge, say $(u,v)$. Let us denote the edge of $P_i$ as $(u,v)$ and the same edge in $P_j$ as $(u',v')$. Then the following two cases arise:
case 1: $u =u' \& v = v'$, case 2: $u =v' \& v = u'$ ( for example, if an edge $(u,v) = (1,2)$ in $P_i$ is appearing in another path $P_j$, then it can come in the order $(1,2)$ or $(2,1)$ in $P_j$).  These cases, along with the odd/even position of vertex $u$ give rise to the following sub cases: \\

\begin{table}[H]
    \centering
    \begin{tabular}{|p{3.4cm}|p{3.4cm}|}
    \hline
        Case I ($u=u'$, $v=v'$)  & Case II ($u=v'$, $v=u'$) \\ \hline
        $u$ is in even position and $u'$ is in even position & $u$ is in even position and $v'$ is in even position \\ \hline

        $u$ is in even position and $u'$ is in odd position & $u$ is in even position and $v'$ is in odd position \\ \hline

        $u$ is in odd position and $u'$ is in even position & $u$ is in odd position and $v'$ is in even position \\ \hline

        $u$ is in odd position and $u'$ is in odd position & $u$ is in odd position and $v'$ is in odd position \\ \hline
    \end{tabular}
    \caption{Subcases}
    \label{subcases}
\end{table}

Table \ref{subcases} shows all possible subcases of Case I and Case II depending on the parity of the position of $u$, $u'$ and $v'$. Tables \ref{u=u'} and \ref{u=v'} show the implications of each of the subcases mentioned in Table \ref{subcases}. For example, the subcase $u$ is in even position and $u'$ is in even position is given in the first cell of Table \ref{subcases} (along with $u=u'$, $v=v'$), implies that $i+k \pmod m=j+t \pmod m$ and $i-k+m \pmod m=j-t+m \pmod m$ when $u$ is in $2k$th position and $u'$ is in $2t$th position. For $0 \leq k,t \leq \frac{m}{2}$,

\begin{table}[H]
    \centering
    \begin{tabular}{|c|c|c|}
    \hline
       & $u'$ is in $2t$th position &  $u'$ is in $(2t+1)$th position \\\hline
      $u$ is in $2k$th position  & \begin{tabular}{@{}c@{}}$i+k \pmod m=j+t \pmod m$ \& \\ $i-k+m \pmod m=j-t+m \pmod m$ \end{tabular} & \begin{tabular}{@{}c@{}}$i+k \pmod m =j-t+m \pmod m$  \& \\ $i-k+m \pmod m =j+t+1\pmod m$\end{tabular} \\ \hline
      
      $u$ is in $(2k+1)$th position  & \begin{tabular}{@{}c@{}}$i-k+m\pmod m=j+t\pmod m$ \& \\ $i+k+1\pmod m=j-t+m\pmod m$ \end{tabular} & \begin{tabular}{@{}c@{}}$i-k+m\pmod m=j-t+m\pmod m$ \& \\ $i+k+1\pmod m=j+t+1\pmod m$ \end{tabular} \\ \hline
      
   \end{tabular}
    \caption{$u=u'$ and $v=v'$}
    \label{u=u'}
\end{table}

\begin{table}[H]
    \centering
    \begin{tabular}{|c|c|c|}
    \hline
       & $v'$ is in $2t$th position &  $v'$ is in $(2t+1)$th position \\\hline

      $u$ is in $2k$th position  & \begin{tabular}{@{}c@{}}$i+k\pmod m=j+t\pmod m$ \& \\ $i-k+m\pmod m=j-(t-1)+m\pmod m$\end{tabular} & \begin{tabular}{@{}c@{}}$i+k\pmod m=j-t+m\pmod m$ \& \\ $i-k+m\pmod m=j+t\pmod m$\end{tabular} \\ \hline

       \begin{tabular}{@{}c@{}}$u$ is in $(2k+1)$th \\ position\end{tabular} & \begin{tabular}{@{}c@{}}$i-k+m\pmod m=j+t\pmod m$ \& \\ $i+k+1\pmod m=j-(t-1)+m\pmod m$\end{tabular} & \begin{tabular}{@{}c@{}}$i-k+m\pmod m=j-t+m\pmod m$ \& \\ $i+k+1\pmod m=j+t\pmod m$ \end{tabular} \\ \hline

    \end{tabular}
    \caption{$u=v'$ and $v=u'$}
    \label{u=v'}
\end{table}

If we consider the first cell of Table \ref{u=u'}, that is, $ i+k \pmod m= j+t \pmod m \text{ and } i-k+m \pmod m=j-t+m \pmod m$, we get a contradiction $i=j$. Similarly, if we consider the second cell of Table \ref{u=u'}, that is, $i+k \pmod m= j-t+m \pmod m \text{ and } i-k+m \pmod m=j+t+1 \pmod m$, then we get $2i=2j+1$ as a contradiction. In a similar way, in all other possibilities given in Table \ref{u=u'} and Table \ref{u=v'}, we arrive at either $i=j$ or $2i=2j+1$ (a contradiction). Therefore the paths are edge disjoint and hence $\mathcal{P}$ is a minimum path decomposition of the graph.\\

\end{proof}

\noindent\textbf{Case 2 ($m$ is odd):} In this case, since $m-1$ is an even number, let $\mathcal{P}$ be the minimum path decomposition of $K_{m-1}$ as discussed in Case 1. Now, we modify $\mathcal{P}$ to get a minimum path decomposition of $K_m$. Observe that the path decomposition $\mathcal{P}$ of $K_{m-1}$ contains all the edges of $K_m$ other than the edges incident on the vertex $m$. We modify $\mathcal{P}$ to a minimum path decomposition $\mathcal{P^*}$ of graph $K_m$ as follows:\\
For $1 \le i \le  \frac{m-1}{2}$,
\begin{enumerate}
    \item Let $P'_i$ be a cycle in $K_m$ obtained from $P_i$ by adding the edges $(i,m)$ and $(i+\frac{m-1}{2},m)$.
    \item Remove edge $(i,i+1)$ from $P'_i$ to get a path $P^*_i$ in $K_m$ having the end vertices $i$ and $i+1$. 
    \item Let $P^*_{\frac{m+1}{2}}$ be a path in $K_m$ formed using these deleted edges. That is,  $$P^*_{\frac{m+1}{2}}:1,2,3, \dots, i, i+1, \dots, \frac{m-1}{2}, \frac{m-1}{2}+1.$$
\end{enumerate}

Then, observe that $\mathcal{P^*}= \{P^*_1, P^*_2, \dots, P^*_{\frac{m-1}{2}}, P^*_{\frac{m+1}{2}} \}$ is a path decomposition of $K_m$. Now we prove that $P^*$ is a minimum path decomposition of the complete graph $K_m$ in this case.

\begin{claim}
    $\mathcal{P^*}$ is a minimum path decomposition of $K_m$.
\end{claim}

\begin{proof}

Let $\mathcal{P^{**}}$ be another path decomposition of $K_m$ with $|\mathcal{P^{**}}| < |\mathcal{P^*}|= \frac{m+1}{2}$. Then $\mathcal{P^{**}}$ can contain at most $\frac{m-1}{2}$ paths from $K_m$. Since the maximum length of a path in $K_m$ is $m-1$, observe that $P^{**}$ covers at most $\frac{(m-1)(m-1)}{2}$ edges, which is strictly less than the number of edges of $K_m$. This contradicts the assumption that $\mathcal{P^{**}}$ is a path decomposition of $K_m$. Hence $\mathcal{P^*}$ is a minimum path decomposition of $K_m$.

\end{proof}

From \textbf{Case 1} and \textbf{Case 2}, we conclude that the path number of $K_m$ is $\lceil \frac{m}{2} \rceil$.

\subsection{Minimum path decomposition of $L_{1}(m,2)$}

To construct a minimum path decomposition of $L_1(m,2)$, we use the minimum path decomposition of $K_m$ discussed in Section \ref{subsec4.1}. 

\begin{observation} \label{subd} \textbf{(Subdivision of a path $P_i$ in $K_m$ to get a path $P'_i$ of $L_1(m,2)$):}\\

Replace each edge $(u,v) \in E(P_i)$ with a path $u, \{u,v\}, v$ in $L_1(m,2)$ to get a path $P'_i$ in $L_1(m,2)$. For example, $P_1: 1, 2, 6, 3, 5, 4$ is a path in $K_6$ and $P'_1: 1,\{1,2\}, 2,\{2,6\}, 6,\{6,3\}, 3,\{3,5\}, 5,\{5,4\}, 4$ is a subdivision of $P_1$ which is a path in $L_1(m,2)$.
\end{observation}

\begin{claim} \label{c3}
    Let $\mathcal{P}= \{P_1,P_2,\dots,P_t\}$ be a path decomposition in $K_m$. Then $\mathcal{P'}= \{P'_1,P'_2,\dots,P'_t\}$ obtained by subdividing paths in $\mathcal{P}$ (see Observation \ref{subd}) is a path decomposition of $L_1(m,2)$. 
\end{claim}

\begin{proof}
   Let $P'_i$ and $P'_j$ be two paths in $\mathcal{P'}$ that share a common edge $g_1$. Since any edge in $L_1(m,2)$ has the form $(u,\{u,v\})$ or $(\{u,v\},v)$, without loss of generality, we take $g_1=(u,\{u,v\})$. As $P'_i$ and $P'_j$ are obtained by subdividing $P_i$ and $P_j$ of $K_m$, note that the above assumption implies that the edge $(u,v)$ is present in both $P_i$ and $P_j$ of $K_m$. This contradicts $\mathcal{P}$ is a path decomposition of $K_m$. Hence $P'_i$ and $P'_j$ are distinct for any $i\neq j$. Observe that after subdividing each edge $(u,v)$ in path $P_i \in \mathcal{P}$ to get the path $P'_i \in \mathcal{P'}$, the length of $P'_i$ becomes $2$(length of $P_i$). Since $\mathcal{P'}$ contains $\frac{m}{2}$ edge disjoint paths, all of these paths together cover $\frac{m}{2}(2m-2) = m(m-1) = |E(L_1(m,2)|$ edges. Hence $\mathcal{P'}$ is a path decomposition of $L_1(m,2)$. 
   
\end{proof}

\begin{claim} \label{c4}
    If $\mathcal{D}$ is a path decomposition of $L_1(m,2)$, then $|\mathcal{D}| \geq \lfloor \frac{m}{2} \rfloor$. 
\end{claim}

\begin{proof}
    We prove this claim by contradiction. Suppose that the path number of $L_1(m,2)$ is less than $\lfloor \frac{m}{2} \rfloor$. Then, there is a path decomposition of size at most $ \lfloor \frac{m}{2} \rfloor-1$. Note that in a bipartite graph with vertex partition $(A,B)$, where $|A|=m_1$ and $|B|=m_2$, the maximum length path that we are able to take is $2\times min \{m_1,m_2\}$. In $L_1(m,2)$, except for the case $m=2$, we have $m_1 \leq m_2$, where $m_1=m$ and $m_2=\frac{m(m-1)}{2}$. Therefore, the length of the path is at most $2\times min \{m_1,m_2\}= 2m$. Hence, we can use at most $m(m-2)$ edges,  which is less than the total number of edges of $L_1(m,2)$. This means that there are $m$ edges which are still uncovered. Hence, the size of a path decomposition of $K_m$ should be greater than $ \lfloor \frac{m}{2} \rfloor-1$. For the case $m=2$, it is obvious. 
\end{proof}
Now, we prove Theorem \ref{m2} which discusses the path number of $L_1(m,2)$, for $m \geq 2$. 
\begin{proof}[{\bf Proof of  Theorem~\ref{m2}}]
Consider the following cases:\\
\noindent\textbf{Case 1 ($m$ is even):} In this case, we use the path decomposition $\mathcal{P}$ of $K_m$ given in section \ref{subsec4.1} (Case 1) to find a minimum path decomposition of $L_1(m,2)$. We subdivide each path in $\mathcal{P}$ to get a path decomposition $\mathcal{P'}$ of $L_1(m,2)$ as given in  Observation \ref{subd}. Then by Claim \ref{c3} and Claim \ref{c4} the path decomposition $\mathcal{P'}$ is a minimum path decomposition of $L_1(m,2)$.\\

\noindent\textbf{Case 2 ($m$ is odd):} In this case, we use the same subdividing technique as in case 1.  When $m$ is odd, observe that the path decomposition $\mathcal{P}$ of $K_m$ (see section \ref{subsec4.1}, Case 2) contains $\frac{m+1}{2}$ paths where $\frac{m-1}{2}$ are Hamiltonian paths and one non-Hamiltonian path of $K_m$. Similar to Case 1, we subdivide each edge $(u,v)$ to a path $u, \{u,v\}, v$ to get a path $P^*_i$ of $L_1(m,2)$ and let $\mathcal{P^*}= \{P^*_1,P^*_2,\dots,P^*_\frac{m+1}{2}\}$. After subdividing, the length of each path $P^*_i$ for $1\leq i\leq \frac{m-1}{2}$ becomes $2(m-1)$ and the length of $P^*_\frac{m+1}{2}$ becomes $(m-1)$. Now, we remove $P^*_\frac{m+1}{2}$ from $\mathcal{P^*}$ by replacing the edges of the path $P^*_{\frac{m+1}{2}}:1,\{1,2\}, 2,\{2,3\},3,\dots,\frac{m-1}{2},\{\frac{m-1}{2},\frac{m+1}{2}\},\frac{m+1}{2}$ to specific paths  $P^*_i$ for $1\leq i \leq \frac{m-1}{2}$. Observe that for $1\leq i \leq \frac{m-1}{2}$, the vertices $i$ and $i+1$ are the end vertices of $P^*_i$. We attach the edges of the path $P^*_{\frac{m+1}{2}}$ in the following way.\\

\textbf{When $\frac{m-1}{2}$ is even:}\\
For $1 \le i \le \frac{m-1}{2}$, perform the following actions to get path $P_i'$ from the path $P_i^* \in \mathcal{P^*}$:

\textit{When $i$ is odd}: Remove the edges $(i, \{i, i+1\})$ and $(i+1, \{i+1, i+2\})$ from $P^*_\frac{m+1}{2}$ then attach $(i, \{i, i+1\})$ to the vertex $i$ and $(i+1, \{i+1, i+2\})$ to the vertex $i+1$ in $P_i^*$.\\

\textit{When $i$ is even}: Remove the edges $(i, \{i, i-1\})$ and $(i+1, \{i+1, i\})$ from $P^*_\frac{m+1}{2}$ then attach $(i, \{i, i-1\})$ to the vertex $i$ and $(i+1, \{i+1, i\})$ to the vertex $i+1$ in $P_i^*$.\\

\textbf{When $\frac{m-1}{2}$ is odd:}\\
Except for the last three paths, that is, for $1\leq i \leq \frac{m-7}{2}$, we use the same rule mentioned in the above `$\frac{m-1}{2}$ is even' case to find the path $P'_i$ from $P_i^*$. In the path $P^*_\frac{m-5}{2}$, we attach the edge $(\frac{m-5}{2},\{\frac{m-5}{2}, \frac{m-3}{2}\})$ to the vertex $\frac{m-5}{2}$ and the edge $(\{\frac{m-3}{2}, \frac{m-1}{2}\}, \frac{m-3}{2})$ to the vertex $\frac{m-3}{2}$ to form the path $P'_\frac{m-5}{2}$. Now attach the edge $(\frac{m-3}{2},\{\frac{m-5}{2}, \frac{m-3}{2}\})$ to the vertex $\frac{m-3}{2}$ and the edge $(\{\frac{m-1}{2}, \frac{m+1}{2}\}, \frac{m-1}{2})$  to the vertex $\frac{m-1}{2}$ of  the path $P^*_\frac{m-3}{2}$  to form the path $P'_\frac{m-3}{2}$. In the final path $P^*_\frac{m-1}{2}$, we can attach the edge $(\frac{m-1}{2},\{\frac{m-3}{2}, \frac{m-1}{2}\})$ to the vertex $\frac{m-1}{2}$ and the edge $(\{\frac{m-1}{2}, \frac{m+1}{2}\}, \frac{m+1}{2})$ to the vertex $\frac{m+1}{2}$ to form the path $P'_\frac{m-1}{2}$.

Observe that in both cases for each $i$, we have removed two edges from the path $P^*_{\frac{m+1}{2}}$ and attached to $P_i^*$ to get the path $P_i'$. Hence the $\frac{m-1}{2}$ new paths together accommodate all edges of the path $P_{\frac{m+1}{2}}$. This implies that there is a path decomposition $\mathcal{P}' = \{P_1', P_2' \dots, P'_{\frac{m-1}{2}}\}$ of size $\lfloor \frac{m}{2} \rfloor$ when $m$ is odd.

\end{proof}

\section{Conclusion}
In this work, we have proved Gallai's conjecture for Levi graph of order one $L_{1}(m,k)$ for all $ m \ge 2 $ and for all $2 \le k \le m$. We have also developed a step-by-step method to determine the path number of $L_{1}(m,2)$. For $1 < t < k$, the graph $L_{t}(m,k)$ is a bipartite graph where one part contains all the $(k-t)$-elements subsets and other part contains all the $k$-elements subsets of $[m]= \{1, 2, \dots, m\}$. A vertex representing $(k-t)$-elements subsets is adjacent to a vertex representing $k$-elements subsets if and only if the $(k-t)$-elements subset is properly contained in the $k$-elements subsets. In this discussion, we have focused solely on the Levi graph of order one. However, this work can be extended to the Levi graph of order $t$, denoted by $L_{t}(m,k)$.

\noindent To prove Gallai's conjecture for the Levi Graph of order one, we utilize the degree condition: if the degree of a vertex is odd, then the vertex is an end vertex of some path. Because of this degree condition, we reintroduce the crossing edges to that vertex. In this scenario, the set of crossing edges forms a matching, which simplifies the process. However, for $L_{t}(m,k)$, the set of crossing edges does not form a matching. This complicates the application of the same technique. Consequently, we need to develop an alternative approach to address this issue effectively.\\
\noindent The next step is to determine the path number. In our discussion, we have developed a step-by-step method to determine the path number of $L_{1}(m,2)$ for $m \ge 2 $. This method utilizes an auxiliary graph, which in this case is a complete graph. But for $k > 2$, this auxiliary graph is not a simple graph at all. Hence the similar method will not work here. Therefore we have to use a different method to find the path number.\\

\noindent As an extension of this work, we have suggested the following three problems.\\

\noindent \textbf{Problem 1.} Prove or disprove that the Levi graph $L_{t}(m,k)$ on $n$ vertices has a path decomposition of size at most $\lceil \frac{n}{2} \rceil$.\\
\noindent \textbf{Problem 2.} Determine the path number of $L_{1}(m,k)$, for $k > 2$.\\
\noindent \textbf{Problem 3.} Determine the path number of $L_{t}(m,k)$, for $t > 1$ and $3 \le k \le m$.\\

\end{document}